\documentclass[a4paper,11pt]{article}
\usepackage[margin=0.8in]{geometry}
\usepackage{enumitem}

\usepackage{braket}
\usepackage[backref=page,hidelinks]{hyperref}
\usepackage{amssymb,mathtools,amsthm,amsmath}
\usepackage{tikz}

\usepackage{xargs}

\usepackage{verbatim}

\theoremstyle{plain}
\newtheorem{theorem}{\bf Theorem}

\newtheorem{conjecture}[theorem]{\bf Conjecture}
\newtheorem{problem}[theorem]{\bf Problem}
\newtheorem{proposition}[theorem]{\bf Proposition}
\newtheorem{corollary}[theorem]{\bf Corollary}
\newtheorem{lemma}[theorem]{\bf Lemma}
\theoremstyle{definition}

\newenvironment{remark}[1][Remark.]{\begin{trivlist}
		\item[\hskip \labelsep {\bfseries #1}]}{\end{trivlist}}

\numberwithin{theorem}{section}
\numberwithin{equation}{section}

\newcommand{\Rea}{{\mathbb R}}

\DeclareMathOperator{\eps}{\varepsilon}
\DeclareMathOperator{\teps}{\tilde{\varepsilon}}

\providecommand{\keyword}[1]
{
  {\small	
  \textbf{\textit{Keywords:~}} #1}
}

\usepackage[d]{esvect}

\usepackage[nobysame,msc-links,non-sorted-cites,abbrev]{amsrefs}

\renewcommand{\eprint}[1]{\href{https://arxiv.org/abs/#1}{arXiv:#1}}

\BibSpec{book}{%
	+{}  {\PrintPrimary}                {transition}
	+{,} { \textbf}                     {title} 
	+{.} { }                            {part}
	+{:} { \textit}                     {subtitle}
	+{,} { \PrintEdition}               {edition}
	+{}  { \PrintEditorsB}              {editor}
	+{,} { \PrintTranslatorsC}          {translator}
	+{,} { \PrintContributions}         {contribution}
	+{,} { }                            {series}
	+{,} { \voltext}                    {volume}
	+{,} { }                            {publisher}
	+{,} { }                            {organization}
	+{,} { }                            {address}
	+{,} { \PrintDateB}                 {date}
	+{,} { }                            {status}
	+{}  { \parenthesize}               {language}
	+{}  { \PrintTranslation}           {translation}
	+{;} { \PrintReprint}               {reprint}
	+{.} { }                            {note}
	+{.} {}                             {transition}
	+{}  {\SentenceSpace \PrintReviews} {review}
}
\BibSpec{incollection}{%
  +{}  {\PrintAuthors}                {author}
  +{,} { \textit}                     {title}
  +{.} { }                            {part}
  +{:} { \textit}                     {subtitle}
  +{,} { \PrintContributions}         {contribution}
  +{,} { \PrintConference}            {conference}
  +{}  {\PrintBook}                   {book}
  +{,} { }                            {booktitle}
	+{}  { \PrintEditorsB}              {editor}
	+{,} { }                            {publisher}
  +{,} { \PrintDateB}                 {date}
  +{,} { pp.~}                        {pages}
  +{,} { }                            {status}
  +{,} { \PrintDOI}                   {doi}
  +{,} { available at \eprint}        {eprint}
  +{}  { \parenthesize}               {language}
  +{}  { \PrintTranslation}           {translation}
  +{;} { \PrintReprint}               {reprint}
  +{.} { }                            {note}
  +{.} {}                             {transition}
  +{}  {\SentenceSpace \PrintReviews} {review}
}

\usepackage{authblk}

\begin{document}

\title{An approximate version of Brouwer's Laplacian conjecture}

    \author{Alan Lew\footnote{
    Faculty of Mathematics, Technion Israel Institute of Technology, Technion City, Haifa 3200003, Israel. E-mail: \href{mailto:alanlew@technion.ac.il}{alanlew@technion.ac.il}.}}

	\date{}
	\maketitle

\begin{abstract}

Let $G=(V,E)$ be an $n$-vertex graph, $L(G)\in \mathbb{R}^{n\times n}$ its Laplacian matrix, and let $\lambda_1(L(G))\ge \lambda_2(L(G))\ge \cdots\ge \lambda_n(L(G))=0$ denote its eigenvalues. For $1\le k\le n$, let $\varepsilon_k(G)= \sum_{i=1}^k \lambda_i(L(G)) -|E|$. We show that for every $1\le k\le n$,
\[
    \varepsilon_k(G) \le \max_{U\subset V,\, |U|=k} |E_G(U)| + (4k-2)\sqrt{k},
\]
where $E_G(U)$ is the set of edges of $G$ contained in $U$. As an immediate consequence, we obtain that $\varepsilon_k(G)\le \binom{k}{2}+(4k-2)\sqrt{k}$. This improves upon previously known bounds for large values of $k$, and may be seen as an approximate version of a conjecture of Brouwer, stating that $\varepsilon_k(G)\le \binom{k+1}{2}$ for every graph $G$. Moreover, for every $r\ge 2$, if $G$ is a $K_{r+1}$-free graph, we obtain that $\varepsilon_k(G)\le (1-1/r)k^2/2 + (4k-2)\sqrt{k}$, which is tight up to the sub-quadratic term.

Our arguments rely on the study of the largest eigenvalue of a matrix obtained by performing a certain diagonal perturbation on the $k$-th additive compound matrix of $L(G)$. 
Using similar methods, we show that the largest Laplacian eigenvalue of the $k$-th token graph of a graph $G=(V,E)$ is bounded from above by $|E|+4k-2$, obtaining a weak version of a conjecture of Apte, Parekh, and Sud, which predicts that an upper bound of $|E|+k$ should hold. 
All our results also hold, with essentially the same proofs, when the Laplacian matrix is replaced by the signless Laplacian of the graph.

\end{abstract}
\keyword{Laplacian, sum of eigenvalues, additive compound, Brouwer's conjecture, token graph}

\section{Introduction}

For a symmetric matrix $M\in \Rea^{m\times m}$, we denote by $\lambda_i(M)$ the $i$-th largest eigenvalue of $M$, for all $1\le i\le m$. 
Let $G=(V,E)$ be an $n$-vertex graph. The \emph{Laplacian matrix} of $G$, denoted by $L(G)$, is an $n\times n$ matrix defined by
\begin{equation}\label{eq:laplacian}
    L(G)_{u,v}=\begin{cases}
        \deg(u) & \text{if } u=v,\\
        -1 & \text{if } \{u,v\}\in E,\\
        0 & \text{otherwise,}
    \end{cases}
\end{equation}
for every $u,v\in V$, where $\deg(u)$ is the degree of $u$ in $G$, that is, the number of edges of $G$ incident to $u$. It is easy to check that $L(G)$ is symmetric and positive semi-definite, and its smallest eigenvalue is equal to $0$ (as the sum of each row of $L(G)$ vanishes). 

The Laplacian matrix was introduced by Kirchhoff in the nineteenth century in his study of electrical networks \cite{kirchhoff1847ueber}. The study of graph Laplacians was further developed, independently, by Fiedler~\cite{fiedler1973algebraic}, Anderson and Morley \cite{anderson1985eigenvalues}, and Kelmans (see \cite{hammer1996laplacian} and the references within).  
In the decades since, the study of the Laplacian matrix and its eigenvalues has become a central theme in spectral graph theory (see, for example, \cites{merris1994survey,mohar1991survey}), with applications to diverse topics in mathematics, computer science, and other fields, such as clustering algorithms \cite{vonluxburg20007tutorial}, convergence of random walks on graphs \cite{caputo2010aldous}, and the control of multi-agent systems \cite{olfati2007consensus}.

The Laplacian eigenvalues of a graph are known to be closely related to a wide range of graph properties, such as connectivity and expansion parameters \cites{fiedler1973algebraic, mohar1989isoperimetric, alon1985lambda1}, the matching number \cites{brouwer2005matching,gu2022matching}, and even to  topological properties of the graph  \cites{aharoni2005eigenvalues,lew2024laplacian}. In \cite{grone1994laplacian}, Grone and Merris initiated the study of the relations between partial sums of Laplacian eigenvalues and combinatorial properties of graphs. In particular, they showed that the sum of the $k$ largest eigenvalues of the Laplacian matrix of a graph $G$ is bounded from below by the sum of the $k$ largest degrees of $G$. On the other hand, they conjectured, and Bai later proved in \cite{bai2011gronemerris}, that the sum of the $k$ largest Laplacian eigenvalues is bounded from above by the sum of the $k$ largest elements in the conjugate degree sequence of $G$.

For a graph $G=(V,E)$, and $1\le k\le |V|$, let us denote $\eps_k(G)=\sum_{i=1}^k \lambda_i(L(G)) - |E|$. Motivated by Grone and Merris' conjecture, Brouwer proposed the following conjecture.

\begin{conjecture}[Brouwer \cites{brouwer2012book,haemers2010onthesum}]
\label{conj:brouwer}
    Let $G=(V,E)$ be a graph, and let $1\le k\le |V|$. Then,
    \[
    \eps_k(G)\le \binom{k+1}{2}.
    \]
\end{conjecture}

Since its introduction, Conjecture \ref{conj:brouwer} has been the focus of much work. The cases $k=1$ and $k=2$ of the conjecture were proved by Haemers, Mohammadian, and Tayfeh-Rezaie in \cite{haemers2010onthesum}. The conjecture has been shown to hold for some special classes of graphs, such as trees \cite{haemers2010onthesum}, regular graphs \cites{ berndsen2010thesis, mayard2010thesis}, split graphs  \cites{ berndsen2010thesis,mayard2010thesis}, connected graphs with maximum degree at most $\sqrt{|V|}/2$ \cite{knill2025remarks}, and graphs with girth at least five \cite{lew2025sums}. For further work around Brouwer's conjecture, see for example \cites{du2012upper,wang2012onaconjecture,das2015energy,chen2018note,chen2019onbrouwers,ganie2020further,rocha2020aas,cooper2021constraints,torres2024brouwer,torres2026critical}.

In previous work \cites{lew2024partition,lew2025sums}, we obtained various weak versions of Conjecture \ref{conj:brouwer}. Particularly, in \cite[Theorem 1.8]{lew2025sums}, we proved that for every $k\ge 1$ and every graph $G=(V,E)$ with $|V|\ge k$, $\eps_k(G)\le k^2$.

For a graph $G=(V,E)$ and $U\subset V$, we denote $E_G(U)=\{e\in E : \, e\subset U\}$. Our main result is the following upper bound on $\eps_k(G)$. 

\begin{theorem}\label{thm:main_density_version}
Let $G=(V,E)$ be a graph, and let $1\le k\le |V|$. Then,
\[
    \eps_k(G) \le \max_{U\subset V,\, |U|=k} |E_G(U)| + (4k-2)\sqrt{k}.
\]
\end{theorem}

As an immediate consequence, we obtain the following result.

\begin{theorem}\label{thm:main_brouwer}
Let $G=(V,E)$ be a graph, and let $1\le k\le |V|$. Then,
\[
    \eps_k(G) \le \binom{k}{2}+ (4k-2)\sqrt{k}.
\]
\end{theorem}

This improves upon the best previously known bound $\eps_k(G)\le k^2$ for all $k\ge 61$, and resolves a problem we posed in \cite[Problem 5.3]{lew2024partition}, asking for an ``approximate" version of Conjecture \ref{conj:brouwer}, which may be stated as follows.
\begin{corollary}\label{cor:approximate_brouwer}
    For every $\epsilon>0$ and $k\ge 64/\epsilon^2$, every graph $G=(V,E)$ with $|V|\ge k$ satisfies
    \[
        \eps_k(G)\le (1+\epsilon)\frac{k^2}{2}.
    \]
\end{corollary}

Recall that a graph $G$ is called $H$-free if it does not contain a copy of the graph $H$ as a subgraph. We denote by $K_r$ the complete graph on $r$ vertices. By Tur\'an's Theorem (see, for example, \cite[Chapter 4]{vanlint2001course}), the following extension of Theorem \ref{thm:main_brouwer} follows immediately from Theorem \ref{thm:main_density_version}.

\begin{theorem}\label{thm:approximate_brouwer_turan}
  Let $r\ge 2$. Let $G=(V,E)$ be a $K_{r+1}$-free graph, and let $1\le k\le |V|$. Then,
\[
    \eps_k(G) \le \left(1-\frac{1}{r}\right)\frac{k^2}{2}+ (4k-2)\sqrt{k}.
\]  
\end{theorem}
The bound in Theorem \ref{thm:approximate_brouwer_turan} is tight up to the sub-quadratic term in $k$ (see Section \ref{sec:concluding_remarks} for details). Note that Theorem \ref{thm:approximate_brouwer_turan} implies that, for every $r\ge 2$, the bound in Brouwer's conjecture holds for every $K_{r+1}$-free graph for all $k\ge 64 r^2$.

For a matrix $M\in \Rea^{m\times m}$ and $1\le k\le m$, the \emph{$k$-th additive compound} of $M$ is a matrix $M^{[k]}\in \Rea^{\binom{m}{k}\times \binom{m}{k}}$ whose spectrum consists of all possible sums of $k$ distinct eigenvalues of $M$ (see Section \ref{sec:compounds} for details). The proof of Theorem \ref{thm:main_density_version} relies on the analysis of the spectral radius (that is, the largest eigenvalue) of a matrix $M_k(G)\in \Rea^{\binom{|V|}{k}\times \binom{|V|}{k}}$, which is obtained by subtracting certain diagonal matrix from the $k$-th additive compound of $L(G)$ (see Section \ref{sec:M_k} for a precise definition).

Let $G=(V,E)$ be a graph, and let $0\le k\le |V|$. The \emph{$k$-th token graph} of $G$, denoted by $F_k(G)$, is the graph on vertex set $\binom{V}{k}=\{\sigma\subset V:\, |\sigma|=k\}$ with edge set
\[
    E_k= \left\{ \{\sigma,\tau\}:\, \sigma\triangle \tau\in E\right\},
\]
where $\sigma\triangle\tau$ denotes the symmetric difference of $\sigma$ and $\tau$. 
Token graphs have been extensively studied in various contexts, often under different names (see, for example, \cites{johns1988generalized,SPITZER1970246,alavi1991double,wright1992n,audenaert2007symmetric,fabila2012token}). 
Note that $F_0(G)$ is just a one-vertex graph, $F_1(G)\cong G$, and, for every $0\le k\le |V|$, $F_k(G)\cong F_{|V|-k}(G)$. 
The study of the Laplacian spectrum of token graphs has received increasing interest in recent years, in works such as \cite{caputo2010aldous, dalfo2021laplacian, dalfo2022algebraic, dalfo2023some, lew2024token}.

Motivated by applications to the analysis of approximation ratios of certain quantum algorithms, Apte, Parekh, and Sud proposed in \cite{apte2025conjectured} the following conjecture.

\begin{conjecture}[Apte, Parekh, Sud {\cite{apte2025conjectured}}]\label{conj:token}
Let $G=(V,E)$ be a graph and let $1\le k\le |V|/2$. Then,
\[
    \lambda_1(L(F_k(G)))\le |E|+k.
\]
\end{conjecture}

The structure of the matrix $L(F_k(G))$ is very similar to that of $L(G)^{[k]}$ --- the $k$-th additive compound of $L(G)$, whose largest eigenvalue is exactly $\sum_{i=1}^k\lambda_i(L(G))$. Hence, Conjecture \ref{conj:token} may be seen as a natural analogue of Brouwer's conjecture in the context of token graphs.

Here, we prove the following weak version of Conjecture \ref{conj:token}.

\begin{theorem}\label{thm:main_token}
 Let $G=(V,E)$ be a graph and let $1\le k\le |V|/2$. Then,\[    \lambda_1(L(F_k(G)))\le |E|+4k-2.\]  
 Moreover, if $G$ is bipartite, then $ \lambda_1(L(F_k(G)))\le |E|+2k-1$.
\end{theorem}

Let $G=(V,E)$ be an $n$-vertex graph. The \emph{signless Laplacian} of $G$, denoted by $Q(G)\in \Rea^{n\times n}$, is the matrix defined by
\[
    Q(G)_{u,v}=\begin{cases}
        \deg(u) & \text{if } u=v,\\
        1 & \text{if } \{u,v\}\in E,\\
        0 & \text{otherwise,}
    \end{cases}
\]
for all $u,v\in V$. Ashraf, Omidi, and Tayfeh-Rezaie proposed in \cite{ashraf2013onthesum} a signless Laplacian analogue of Brouwer's conjecture; namely, they conjectured that $\sum_{i=1}^k \lambda_i(Q(G))\le |E|+ \binom{k+1}{2}$ for every graph $G=(V,E)$ and $1\le k\le |V|$. Similarly, Apte, Parekh, and Sud proposed in \cite{apte2025conjectured} a signless Laplacian version of Conjecture \ref{conj:token}, namely that $\lambda_1(Q(F_k(G)))\le |E|+k$ for all $G=(V,E)$ and $1\le k\le |V|/2$. Let us note that all the results proved in this paper extend, with essentially the same proofs, to the context of signless Laplacian matrices. For simplicity of exposition, we omit the details.

The paper is organized as follows. In Section \ref{sec:preliminaries}, we present some background material on symmetric matrix eigenvalues and on additive compound matrices, which we will later need. In Section \ref{sec:subadditive}, we introduce the notion of \emph{subadditive functions} on graphs (namely, functions on graphs that are subadditive with respect to edge-disjoint unions, of which $\eps_k(G)$ is an example) and present some basic results on them, extending known facts about $\eps_k(G)$. In Section \ref{sec:M_k} we introduce the matrix $M_k(G)$, determine its relation to the function $\eps_k(G)$, and establish some of its basic properties. In Section \ref{sec:M_k_stars_matchings}, we prove some results about the $M_k$-eigenvalues of star graphs and matching graphs, and use them to finish the proof of Theorem \ref{thm:main_density_version}. In Section \ref{sec:token}, we deal with the Laplacian eigenvalues of token graphs, establishing Theorem \ref{thm:main_token}. We close with Section \ref{sec:concluding_remarks}, in which we present some related open problems and conjectures.

\section{Preliminaries}\label{sec:preliminaries}

\subsection{An eigenvalue inequality}

We will need the following classical result due to Weyl.
\begin{lemma}[Weyl (see, for example, {\cite[Theorem 2.8.1]{brouwer2012book}})]
    \label{lemma:weyl}
    Let $A,B\in \Rea^{m\times m}$ be symmetric matrices. Then,
    \[
        \lambda_1(A+B)\le \lambda_1(A)+\lambda_1(B).
    \]
\end{lemma}

\subsection{Additive compound matrices}\label{sec:compounds}

Let $m\ge 1$ and $1\le k\le m$. Let $[m]=\{1,2,\ldots,m\}$. For two sets $\sigma,\tau\in \binom{[m]}{k}=\{\eta\subset [m]:\, |\eta|=k\}$ with $|\sigma\cap \tau|=k-1$, denote $\text{sign}(\sigma,\tau)=(-1)^{|\{ r\in \sigma\cap \tau,\, i<r<j\}|}$, where $i<j$ are the two unique elements in the symmetric difference $\sigma\triangle \tau$.  

Let $M\in \Rea^{m\times m}$ be a matrix. For $1\le k\le m$,  the \emph{$k$-th additive compound} of $M$ is the matrix $M^{[k]}\in \Rea^{\binom{m}{k}\times \binom{m}{k}}$ defined by
\begin{equation}\label{eq:additive_compound}
    M^{[k]}_{\sigma,\tau}=\begin{cases}
        \sum_{i\in \sigma} M_{i,i} & \text{if } \sigma=\tau,\\
        \text{sign}(\sigma,\tau)\cdot M_{i,j} & \text{if } |\sigma\cap \tau|=k-1,\, \sigma\setminus \tau=\{i\},\, \tau\setminus\sigma=\{j\},\\
        0 & \text{otherwise,}
    \end{cases}
\end{equation}
for all $\sigma,\tau\in \binom{[m]}{k}$.

The additive compound operator was studied, for example, in \cite{wielandt1967topics,schwarz1970totally, london1976derivations, fiedler1974additive}. As its name suggests, the $k$-th additive compound operator is additive, that is, $(A+B)^{[k]}=A^{[k]}+B^{[k]}$ for every $A,B\in \Rea^{m\times m}$. Moreover, the eigenvalues of the $k$-th additive compound of a matrix $M$ consist of all possible sums of $k$ distinct eigenvalues of $M$. More precisely, we have the following result.

\begin{lemma}[See, for example, {\cite[Thm. F.5]{marshall1979inequalities}, \cite[Thm. 2.1]{fiedler1974additive}}]\label{lemma:additive_compound}
    Let $M\in \Rea^{m\times m}$ be a matrix with eigenvalues $\lambda_1,\ldots,\lambda_m$. Then, the eigenvalues of $M^{[k]}$ are $\lambda_{i_1}+\cdots+\lambda_{i_k}$, for all $1\le i_1<\cdots<i_k\le m$.
\end{lemma}

\section{Subadditive functions on graphs}\label{sec:subadditive}

Here, we introduce some terminology and basic results which we will later use. 
Let $V$ be a finite set, and let $\mathcal{G}(V)$ be the collection of all graphs on vertex set $V$. We say that a function $f:\mathcal{G}(V)\to \Rea$ is \emph{subadditive} if for every pair of edge-disjoint graphs $G_1, G_2\in \mathcal{G}(V)$, $f(G_1\cup G_2)\le f(G_1)+f(G_2)$.

We will need the following lemma, appearing implicitly, in the special case of $f(G)=\eps_k(G)$, in the work of Haemers, Mohammadian, and Tayfeh-Rezaie \cite{haemers2010onthesum} (see also \cite[Lemma 2.6]{lew2024partition} or \cite[Corollary 2.3]{wang2012onaconjecture} for an explicit statement in the case $f(G)=\eps_k(G)$).

\begin{lemma}\label{lemma:good_subgraph_subadditive}
    Let $V$ be a finite set, and let $f:\mathcal{G}(V)\to \Rea$ be a subadditive function. Then, for every graph $G=(V,E)$ there exists a spanning subgraph $G'=(V,E')$ such that $f(G')\ge f(G)$ and $f(H)>0$ for every non-empty spanning subgraph $H$ of $G'$.
\end{lemma}
\begin{proof}
    We argue by induction on $|E|$. If $|E|=0$, the claim follows trivially (choosing $G'=G$). Assume $|E|>0$. If $f(H)>0$ for every non-empty spanning subgraph $H$ of $G$, we are done (again, choosing $G'=G$). Otherwise, let $H$ be a non-empty spanning subgraph of $G$ with $f(H)\le 0$, and let $G'$ be the graph obtained from $G$ by removing the edges of $H$. By the induction hypothesis, there is a spanning subgraph $G''$ of $G'$ with $f(H')>0$ for every non-empty spanning subgraph $H'$ of $G''$, such that $f(G'')\ge f(G')$. Note that $G''$ is a spanning subgraph of $G$ as well, and, by the subadditivity of $f$, we have
    \[
        f(G)\le f(G')+f(H) \le f(G'')+0=f(G''),
    \]
    as required.
\end{proof}

Let $G=(V,E)$ be a graph. We say that $U\subset V$ is a \emph{cover} of $G$ if $U\cap e\ne \emptyset$ for all $e\in E$.  The \emph{covering number} of $G$, denoted by $\tau(G)$, is the smallest size of a cover in $G$.
The following lemma was proved, in the special case $f(G)=\eps_k(G)$, by Das, Mojallal, and Gutman in \cite{das2015energy} (see also \cite[Corollary 2.10]{lew2024partition}). Our argument here is similar to the one in \cite{das2015energy}.

Let us denote by $\mathcal{S}(V)\subset \mathcal{G}(V)$ the collection of graphs on vertex set $V$ whose edges form a star (that is, every $S\in\mathcal{S}(V)$ is a star graph plus, possibly, some isolated vertices).

\begin{lemma}\label{lemma:subadditive_cover_bound}
    Let $V$ be a finite set, and let $f:\mathcal{G}(V)\to \Rea$ be a subadditive function. Let $K\in \Rea$ such that $f(S)\le K$ for every $S\in \mathcal{S}(V)$. Then, for every $G\in \mathcal{G}(V)$,
    \[
        f(G)\le K\cdot \tau(G).
    \]
\end{lemma}
\begin{proof}
         Let $U=\{v_1,\ldots,v_t\}$ be a cover of $G$ of size $t=\tau(G)$. 
    For $1\le i\le t$, let
    \[
        E_i = \{e\in E:\, v_i\in e, \, v_j\notin e \text{ for all } j<i\},
    \]
    and let $G_i=(V,E_i)$. Note that $G_1,\ldots,G_t$ are pairwise edge-disjoint, $G=G_1\cup\cdots\cup G_t$,  and that each $G_i$ belongs to $\mathcal{S}(V)$. Hence, $f(G_i)\le K$ for all $1\le i\le t$, and, by the subadditivity of $f$, we obtain
    \[
        f(G)\le \sum_{i=1}^t f(G_i)\le K\cdot \tau(G),
    \]
    as wanted.
\end{proof}

 Recall that for a graph $G=(V,E)$, its \emph{matching number}, denoted by $\nu(G)$, is the maximum size of a matching in $G$.  For $t\ge 0$, let $\mathcal{M}_t(V)\subset \mathcal{G}(V)$ be the collection of graphs on vertex set $V$ whose edge set consists of $t$ pairwise disjoint edges (that is, every $M\in \mathcal{M}_t(G)$ consists of a matching of size $t$ plus, possibly, some isolated vertices).

\begin{proposition}\label{prop:star_bound}
Let $V$ be a finite set, and let $f:\mathcal{G}(V)\to \Rea$ be a subadditive function. Let $K\in\Rea$ such that $f(S)\le K$ for every $S\in \mathcal{S}(V)$. In addition, let $t\ge 0$ such that $f(M)\le 0$ for all $M\in \mathcal{M}_t(V)$. Then, for every $G\in \mathcal{G}(V)$,
\[
    f(G)\le 2(t-1)K.
\]
Moreover, if $G$ is bipartite, then
\[
    f(G)\le (t-1)K.
\]
\end{proposition}
\begin{proof}
    Let $G=(V,E)$. By Lemma \ref{lemma:good_subgraph_subadditive}, there exists a subgraph $G'=(V,E')$ of $G$ with $f(G)\le f(G')$, such that $f(H)>0$ for every spanning subgraph $H$ of $G'$. In particular, since $f(M)\le 0$ for every $M\in \mathcal{M}_t(V)$, $G'$ does not contain a matching of size $t$. In other words, $\nu(G')\le t-1$. Since $\tau(G')\le 2\nu(G')$, we obtain, by Lemma \ref{lemma:subadditive_cover_bound},
    \[
        f(G)\le f(G')\le  \tau(G')\cdot K \le 2(t-1)K.
    \]
    Similarly, if $G$ is bipartite, then $G'$ is bipartite as well, and so $\tau(G')=\nu(G')$ by K\"onig's theorem (see, for example, \cite[Chapter 5]{vanlint2001course}). Therefore,
    \[
        f(G)\le f(G')\le  \tau(G')\cdot K\le (t-1)K.
    \]
\end{proof}

\begin{remark}
    The proof of Proposition \ref{prop:star_bound} is similar to arguments appearing in \cite{lew2024partition} in the special case $f(G)=\eps_k(G)$. The idea of reducing the problem of bounding $f(G)$ to the case of graphs with bounded matching number was first applied by Haemers, Mohammadian, and Tayfeh-Rezaie \cite{haemers2010onthesum} in the special case $f(G)=\eps_2(G)$.
\end{remark}

\section{The matrix $M_k(G)$}\label{sec:M_k}

For $k\ge 1$ and $G=(V,E)$ with $|V|\ge k$, we define
\[
    M_k(G) = L(G)^{[k]}-D_k(G),
\]
where $D_k(G)\in \Rea^{\binom{|V|}{k}\times \binom{|V|}{k}}$ is a diagonal matrix defined by
\[
    D_k(G)_{\sigma,\sigma}= |E_G(\sigma)|=|\{e\in E:\, e\subset \sigma\}|
\]
for all $\sigma\in \binom{V}{k}$. We define
\[
    \teps_k(G) = \lambda_1(M_k(G))-|E|.
\]
For convenience, we denote $\teps_0(G)=-|E|$. We call $\teps_k(G)$ the \emph{$M_k$-excess} of $G$.

Fix an arbitrary linear order $<$ on $V$. For $\sigma,\tau\in\binom{V}{k}$ with $|\sigma\cap \tau|=k-1$, let $\text{sign}(\sigma,\tau)=(-1)^{|\{w\in \sigma\cap \tau:\, u<w<v\}|}$, where $u<v$ are the two unique vertices in $\sigma \triangle \tau$. 

\begin{lemma}\label{lemma:M_k_formula}
    Let $G=(V,E)$ and $1\le k\le |V|$. Then, for $\sigma,\tau\in \binom{V}{k}$,
    \[
        M_k(G)_{\sigma,\tau}=\begin{cases}
            |\{e\in E:\, e\cap \sigma\ne \emptyset\}|  & \text{if } \sigma=\tau,\\
            -\text{sign}(\sigma,\tau) & \text{if } \sigma\triangle\tau\in E,\\
            0 & \text{otherwise.}
        \end{cases}
    \]
\end{lemma}
\begin{proof}
  By \eqref{eq:laplacian} and \eqref{eq:additive_compound}, we have
  \[
    L(G)^{[k]}_{\sigma,\tau}=\begin{cases}
        \sum_{v\in \sigma} \deg(v) & \text{if } \sigma=\tau,\\
        -\text{sign}(\sigma,\tau) & \text{if } \sigma\triangle \tau\in E,\\
        0 & \text{otherwise,}
    \end{cases}
  \]
  for all $\sigma,\tau\in \binom{V}{k}$. 
  The claim then follows by noting that $M_k(G)_{\sigma,\tau}=L(G)^{[k]}_{\sigma,\tau}$ for all $\sigma\ne \tau$, and that for every $\sigma\in \binom{V}{k}$,
  \[
    \sum_{v\in \sigma} \deg(v)= |\{e\in E:\, |e\cap \sigma|=1\}|+2 |E_G(\sigma)|,
  \]
  and so
  \[
   M_k(G)_{\sigma,\sigma}=\sum_{v\in \sigma} \deg(v)- |E_G(\sigma)|= |\{e\in E:\, |e\cap \sigma|=1\}|+ |E_G(\sigma)| = |\{e\in E:\, e\cap\sigma\ne \emptyset\}|.
  \]
\end{proof}

\begin{lemma}\label{lemma:from_Mk_to_Brouwer}
    Let $G=(V,E)$ be a graph and let $1\le k\le |V|$. Then,
    \[
    \eps_k(G)\le \teps_k(G) + \max_{\sigma\in \binom{V}{k}} |E_G(\sigma)|.
    \]
\end{lemma}
\begin{proof}

By Lemma \ref{lemma:additive_compound},
\[
    \eps_k(G)= \lambda_{1}(L(G)^{[k]})-|E|.
\]
By Lemma \ref{lemma:weyl},
\[
    \lambda_1(L(G)^{[k]})\le \lambda_1(M_k(G))+\lambda_1(D_k(G))
    =\teps_k(G)+|E|+\max_{\sigma\in \binom{V}{k}} |E_G(\sigma)|.
\]
So,
\[
 \eps_k(G)\le \teps_k(G) + \max_{\sigma\in \binom{V}{k}} |E_G(\sigma)|.
\]

\end{proof}

\begin{lemma}\label{lemma:M_k_excess_subadditive}
    Let $k\ge 1$, and let $V$ be a finite set with $|V|\ge k$. Then, the function $\teps_k(G)$ is subadditive on the collection of graphs with vertex set $V$.
\end{lemma}
\begin{proof}
    Let $G_1=(V,E_1)$ and $G_2=(V,E_2)$, where $E_1\cap E_2=\emptyset$. Let $G=G_1\cup G_2$. It is easy to verify, using Lemma \ref{lemma:M_k_formula}, that
    \[
        M_k(G)= M_k(G_1)+M_k(G_2).
    \]
    So, by Lemma \ref{lemma:weyl},
    \[
        \lambda_1(M_k(G))\le \lambda_1(M_k(G_1))+\lambda_1(M_k(G_2)).
    \]
    Therefore, since $|E|=|E_1|+|E_2|$, we obtain
    \[
    \teps_k(G)\le \teps_k(G_1)+\teps_k(G_2).
    \]
\end{proof}

Let $n,m\ge 1$. For matrices $A\in \Rea^{n\times n}$ and $B\in \Rea^{m\times m}$, the \emph{Kronecker product} of $A$ and $B$ is the matrix $A\otimes B\in \Rea^{nm\times nm}$ defined as an $n\times n$ block matrix whose $(i,j)$-th block is the matrix $A_{ij} B\in \Rea^{m\times m}$, for all $1\le i,j\le n$. The \emph{Kronecker sum} of $A$ and $B$ is the matrix $A\oplus B = A\otimes I_{m}+I_{n}\otimes B$, where $I_n,I_m$ are the $n\times n$ and $m\times m$ identity matrices, respectively. It is well-known (see \cite[Theorem 4.4.5]{horn1991topics}) that the eigenvalues of $A\oplus B$ are exactly $\lambda_i(A)+\lambda_j(B)$, for $1\le i\le n$ and $1\le j\le m$.

\begin{lemma}\label{lemma:vertex_disjoint}
    Let $G_1=(V_1,E_1)$ and $G_2=(V_2,E_2)$ be graphs on disjoint vertex sets. Then, for every $1\le k\le |V_1|+|V_2|$,
    \[
        \teps_k(G_1\cup G_2) = \max\left\{\teps_i(G_1)+\teps_j(G_2) :\, 0\le i\le |V_1|,\, 0\le j\le |V_2|,\, i+j=k\right\}.
    \]
\end{lemma}
\begin{proof}

    Fix arbitrary orders on $V_1$ and $V_2$, and let $<$ be a common extension of these orders to $V_1\cup V_2$ such that all vertices in $V_1$ precede all the vertices in $V_2$. 

    Let $G=G_1\cup G_2$, and let $V=V_1\cup V_2$ and $E=E_1\cup E_2$. Let us choose an order on $\binom{V}{k}$, first by ordering the sets $\sigma\in \binom{V}{k}$ increasingly according to $|\sigma\cap V_1|$, then, within each class $\{\sigma:\, |\sigma\cap V_1|=i\}$, according to the lexicographic order on the subsets $\sigma\cap V_1$ induced by the given order on $V_1$, and finally, within each class $\{\sigma:\, \sigma\cap V_1=\eta\}$ according to the lexicographic order on the subsets $\sigma\cap V_2$ induced by the given order on $V_2.$
    
    Note that for $\sigma,\tau\in \binom{V}{k}$, we have $\sigma\triangle \tau\in E$ if and only if $\sigma\cap V_i=\tau\cap V_i$ and $(\sigma\cap V_j)\triangle (\tau\cap V_j)\in E_j$, for $i=1$ and $j=2$, or $i=2$ and $j=1$. Therefore, by Lemma \ref{lemma:M_k_formula}, the matrix $M_k(G)$ is a block diagonal matrix with blocks $M_{ij}\in \Rea^{m_{ij}\times m_{ij}}$, where $m_{ij}=\binom{|V_1|}{i}\binom{|V_2|}{j}$, and the rows and columns of each matrix $M_{ij}$ are indexed by the sets $\sigma\in \binom{V}{k}$ with $|\sigma\cap V_1|=i$ and $|\sigma\cap V_2|=j$, for every $0\le i\le |V_i|$ and $0\le j\le |V_j|$ with $i+j=k$.  Moreover, it is easy to verify, using Lemma \ref{lemma:M_k_formula}, that $M_{ij}= M_i(G_1)\oplus M_j(G_2)$ (where, for convenience, we define $M_0(G_1)$ and $M_0(G_2)$ to be the $1\times 1$ zero matrix). Hence,
    \[
        \lambda_1(M_k(G))= \max\{ \lambda_1(M_i(G_1))+\lambda_1(M_j(G_2)):\, 0\le i\le |V_1|,\, 0\le j\le |V_2|,\, i+j=k\}.
    \]
    Since $|E|=|E_1|+|E_2|$, we obtain
    \[
        \teps_k(G) = \max\left\{\teps_i(G_1)+\teps_j(G_2) :\, 0\le i\le |V_1|,\, 0\le j\le |V_2|,\, i+j=k\right\},
    \]
    as wanted.
\end{proof}

As a consequence, we obtain the following useful result. 
\begin{corollary}\label{corollary:M_k_isolated}
    Let $G=(V,E)$ be a graph, and let $G'$ be obtained from $G$ by adding $t$ isolated vertices. Then, for $1\le k\le |V|+t$,
    \[
    \teps_k(G')= \max_{\max\{0,k-t\}\le i\le \min\{k,|V|\}} \teps_i(G).
    \]
\end{corollary}
\begin{proof}
Let $U$ be a set of size $t$ with $U\cap V=\emptyset$, and let $G_1=(U,\emptyset)$ be the graph consisting of $t$ isolated vertices. It is easy to check that $M_j(G_1)=0$, and so $\teps_j(G_1)=0$, for all $1\le j\le t$. For $j=0$, we have $\teps_0(G_1)=0$ as well, by definition. Since $G'=G\cup G_1$, we obtain, by Lemma \ref{lemma:vertex_disjoint},
    \[
    \teps_k(G')= \max_{\max\{0,k-t\}\le i\le \min\{k,|V|\}} \teps_i(G).
    \] 
\end{proof}

\section{The $M_k$-excess of star graphs and matchings}\label{sec:M_k_stars_matchings}

Let $S_n$ be the star graph on $n$ vertices. That is, it is the graph with vertex set $[n]$ and edge set $\{\{1,i\}:\, 2\le i\le n\}$.

\begin{lemma}\label{lemma:M_k_excess_for_stars}
    Let $n\ge 2$ and $1\le k\le n$. Then,
    \[
        \teps_k(S_n)\le \sqrt{k}.
    \]
\end{lemma}
\begin{proof}
Recall that $M_k(S_n)=L(S_n)^{[k]}-D_k(S_n)$. The matrix $D_k(S_n)$ satisfies
\[
    D_k(S_n)_{\sigma,\sigma}=\begin{cases}
        k-1 & \text{if } 1\in \sigma,\\
        0  & \text{otherwise,}
    \end{cases}
\]
for all $\sigma\in \binom{[n]}{k}$. Let $D\in \Rea^{n\times n}$ be the diagonal matrix defined by
\[
    D_{i,i}=\begin{cases}
        k-1 & \text{if } i=1,\\
        0 & \text{otherwise,}
    \end{cases}
\]
for all $1\le i\le n$. It is easy to check, by \eqref{eq:additive_compound}, that $D_k(S_n)=D^{[k]}$. By the additivity of the additive compound operator, we obtain
\[
    M_k(S_n)= (L(S_n)-D)^{[k]}.
\]
So, by Lemma \ref{lemma:additive_compound},
\[
\lambda_1(M_k(S_n))=\sum_{i=1}^k \lambda_i(L(S_n)-D).
\]
Let $M=L(S_n)-D$. Then,
\[
    M_{i,j}=\begin{cases}
        n-k & \text{if } i=j=1,\\
        1 &\text{if } i=j\ne 1,\\
        -1 & \text{if } i\ne j,\, 1\in\{i,j\},\\
        0 & \text{otherwise,}
    \end{cases}
\]
for all $1\le i,j\le n$. Let $x=\begin{pmatrix}
    y & z^T
\end{pmatrix}^T$, where $y\in \Rea$ and $z\in \Rea^{n-1}$. It is easy to check that for $y=0$ and $z\perp {\bf 1}$ (where ${\bf 1}\in \Rea^{n-1}$ is the all-ones vector), then $x$ is an eigenvector of $M$ with eigenvalue $1$. Therefore, $1$ is an eigenvalue of $M$ with multiplicity $n-2$. 

Now, assume that $y\in\Rea$ and $z= c {\bf 1}$ for some $c\in\Rea$. Then,
\[
    M x= \begin{pmatrix}
        (n-k)y - (n-1)c\\
        (-y+c) {\bf 1}
    \end{pmatrix}.
\]
Therefore, the two remaining eigenvalues of $M$ are the eigenvalues of the $2\times 2$ matrix
\[
    \begin{pmatrix}
        n-k & -n+1\\
        -1 & 1
    \end{pmatrix},
\]
which are
\[
    \lambda_{+,-}= \frac{n-k+1\pm \sqrt{(n-k+1)^2+4(k-1)}}{2}.
\]
For $k=n$, it is easy to check that $\teps_k(S_n)=0$.
For $1\le k\le n-1$, noting that $\lambda_{-}\le 0$, we have
\begin{align*}
    \teps_k(S_n)&= \frac{n-k+1+ \sqrt{(n-k+1)^2+4(k-1)}}{2}+(k-1) -(n-1) \\&= \frac{\sqrt{(n-k+1)^2+4(k-1)}-(n-k-1)}{2}= \frac{\sqrt{(n-k-1)^2+4(n-1)}-(n-k-1)}{2}
    \\&=  \frac{2(n-1)}{\sqrt{(n-k-1)^2+4(n-1)}+(n-k-1)} \\&= 2 \left(\sqrt{\left(1-\frac{k}{n-1}\right)^2+\frac{4}{n-1}}+\left(1-\frac{k}{n-1}\right)\right)^{-1}.
\end{align*}
Let $x=k/(n-1)$. It is easy to verify that, for fixed $k$, the function $\sqrt{(1-x)^2+4 x/k}+(1-x)$ is decreasing, so its minimum on the interval $[0,1]$ is obtained at $x=1$. So, for fixed $k$, the maximum value of $\teps_k(S_n)$ is obtained at $n=k+1$, and is equal to $\sqrt{n-1}=\sqrt{k}$. Hence, $\teps_k(S_n)\le \sqrt{k}$, as wanted.

\end{proof}

Recall that for a finite set $V$, we denote by $\mathcal{S}(V)$ the collection of graphs on vertex set $V$ that are isomorphic to a star graph plus (possibly) some isolated vertices. For $t\ge 0$, we denote by $\mathcal{M}_t(V)$ the collection of graphs on vertex set $V$ whose edge set consists of $t$ pairwise disjoint edges.

\begin{corollary}\label{cor:mk_star}
    Let $k\ge 1$ and $|V|\ge k$. Then, for every $S\in \mathcal{S}(V)$,
    \[
    \teps_k(S)\le \sqrt{k}.
    \]
\end{corollary}
\begin{proof}
    If $S$ has no edges, then $\teps_k(S)=0\le \sqrt{k}$. If $S$ has at least one edge, the claim follows from Corollary \ref{corollary:M_k_isolated} and Lemma \ref{lemma:M_k_excess_for_stars}.
\end{proof}

\begin{lemma}\label{lemma:Mk_for_matchings}
    Let $t\ge k\ge 1$ and let $V$ be a finite set with $|V|\ge 2t$. Let $G\in \mathcal{M}_t(V)$. Then,
    \[
        \teps_k(G)= 2k-t.
    \]
    In particular, if $t\ge 2k$, $\teps_k(G)\le 0$.
\end{lemma}
\begin{proof}
    First, assume that $|V|=2t$, that is, $G$ has no isolated vertices. Let $G=G_1\cup\cdots\cup G_{t}$, where, for $1\le i\le t$, $G_i$ is a two-vertex graph consisting of one of the edges of $G$.  
    It is easy to check that
    \[
        \teps_j(G_i)=\begin{cases}
            -1 & \text{if } j=0,\\
            1 & \text{if } j=1,\\
            0 & \text{if } j=2,
        \end{cases}
    \]
    So, by Lemma \ref{lemma:vertex_disjoint}, we obtain, for $1\le k\le t$,
    \begin{align*}
        \teps_k(G) &= \max\left\{ \sum_{j=1}^t \teps_{i_j}(G_j) :\,  0\le i_1,\ldots,i_t\le 2,\, \sum_{j=1}^t i_j=k\right\}
       \\ &=
        \max\left\{ |\{j\in[t]:\, i_j= 1\}|-|\{j\in[t]:\, i_j= 0\}|     \, :0\le i_1,\ldots,i_t\le 2,\, \sum_{j=1}^t i_j= k\right\} \\&= k-(t-k)= 2k-t.
    \end{align*}
    Note also that $\teps_0(G)=-t$, by definition.
    
    Now, assume $|V|>2t$, and let $1\le k\le t$. By Corollary \ref{corollary:M_k_isolated},
    \[
        \teps_k(G) = \max_{0\le i\le k} (2i-t) =2k-t.
    \]
    In particular, if $t\ge 2k$, then $\teps_k(G)\le 0$, as required.
\end{proof}

Theorem \ref{thm:main_density_version} follows from Lemma \ref{lemma:from_Mk_to_Brouwer} and the following result. 

\begin{proposition}\label{prop:teps_bound}
    Let $G=(V,E)$ be a graph, and let $1\le k\le |V|$. Then,
    \[
        \teps_k(G)\le (4k-2)\sqrt{k}.
    \]
    Moreover, if $G$ is bipartite, then
       \[
        \teps_k(G)\le (2k-1)\sqrt{k}.
    \]
\end{proposition}
\begin{proof}
By Lemma \ref{lemma:M_k_excess_subadditive}, Corollary \ref{cor:mk_star}, Lemma \ref{lemma:Mk_for_matchings}, and Proposition \ref{prop:star_bound}, we obtain, for every graph $G=(V,E)$ and $1\le k\le |V|$, 
\[
    \teps_k(G)\le 2(2k-1)\sqrt{k}=(4k-2)\sqrt{k}.
\]
Moreover, if $G$ is bipartite, then
\[
    \teps_k(G)\le (2k-1)\sqrt{k},
\]
as wanted.
\end{proof}

\begin{remark}
Note that, if $G$ is bipartite, combining Proposition \ref{prop:teps_bound} with Lemma \ref{lemma:from_Mk_to_Brouwer}, we obtain $\eps_k(G)\le k^2/4+(2k-1)\sqrt{k}$. As $k^2/4+(2k-1)\sqrt{k}\le \binom{k+1}{2}$ for $k\ge 59$, we obtain that bipartite graphs satisfy the bound in Brouwer's conjecture for all $k\ge 59$. 
\end{remark}

As mentioned in the introduction, Theorems \ref{thm:main_brouwer} and \ref{thm:approximate_brouwer_turan} follow immediately from Theorem \ref{thm:main_density_version}, while Corollary \ref{cor:approximate_brouwer} follows easily from Theorem \ref{thm:main_brouwer}.

\section{Laplacian spectral radii of token graphs}\label{sec:token}

In this section, we prove Theorem \ref{thm:main_token}. The argument is analogous to the proof of Proposition \ref{prop:teps_bound}.

Let $G=(V,E)$ be a graph. It is easy to verify, using \eqref{eq:laplacian}, that, for $0\le k\le |V|$, we have
\[
    L(F_k(G))_{\sigma,\tau}=\begin{cases}
        |\{e\in E:\, |e\cap \sigma|=1\}| & \text{if } \sigma=\tau,\\
        -1 & \text{if } \sigma\triangle \tau\in E,\\
        0 &\text{otherwise,}
    \end{cases}
\]
for all $\sigma,\tau\in\binom{V}{k}$. Note the similarity between the structure of $L(F_k(G))$ and that of the matrices $L(G)^{[k]}$ and $M_k(G)$. 
For a graph $G=(V,E)$ and $0\le k\le |V|$, let us denote 
\[
\eps^T_k(G)=\lambda_1(L(F_k(G)))-|E|.
\]
First, we state a few simple properties of $\eps_k^T(G)$.

\begin{lemma}\label{lemma:token_subadditive}
    Let $k\ge 1$ and let $V$ be a finite set with $|V|\ge k$. Then, the function $\eps_k^T(G)$ is subadditive on the collection of graphs with vertex set $V$.
\end{lemma}
\begin{proof}
    The proof is identical to the proof of Lemma \ref{lemma:M_k_excess_subadditive}, after noting that $L(F_k(G_1\cup G_2))=L(F_k(G_1))+L(F_k(G_2))$ for every pair of edge-disjoint graphs $G_1=(V,E_1)$ and $G_2=(V,E_2)$.
\end{proof}

The following two results are straightforward analogues of Lemma \ref{lemma:vertex_disjoint} and Corollary \ref{corollary:M_k_isolated}, so we omit their proofs.

\begin{lemma}\label{lemma:token_component_formula}
    Let $G_1=(V_1,E_1)$ and $G_2=(V_2,E_2)$ be graphs on disjoint vertex sets. Then, for every $1\le k\le |V_1|+|V_2|$,
    \[
        \eps_k^T(G_1\cup G_2) = \max\left\{\eps_i^T(G_1)+\eps_j^T(G_2) :\, 0\le i\le |V_1|,\, 0\le j\le |V_2|,\, i+j=k\right\}.
    \]
\end{lemma}

\begin{corollary}\label{cor:token_isolated}
      Let $G=(V,E)$ be a graph, and let $G'$ be obtained from $G$ by adding $t$ isolated vertices. Then, for $1\le k\le |V|+t$,
    \[
    \eps_k^T(G')= \max_{\max\{0,k-t\}\le i\le \min\{k,|V|\}} \eps_i^T(G).
    \]
\end{corollary}

\begin{lemma}\label{lemma:token_excess_for_stars}
    Let $n\ge 2$ and $1\le k\le n$. Then,
    \[
        \eps_k^T(S_n)=1.
    \]
\end{lemma}
\begin{proof}
    It is shown in \cite[Section 5]{dalfo2021laplacian} (see also \cite[Lemma 9]{apte2025conjectured}) that  
    $\lambda_1(L(F_k(S_n)))=n$ for all $1\le k\le n$. Since the number of edges of $S_n$ is $n-1$, we obtain $\eps_k^T(S_n)=1$, as wanted.
\end{proof}

\begin{corollary}\label{cor:token_star_plus_isolated}
    Let $k\ge 1$ and let $V$ be a finite set with $|V|\ge k$. Then, for every $S\in \mathcal{S}(V)$,
    \[
        \eps_k^T(S)\le 1.
    \]
\end{corollary}
\begin{proof}
    If $S$ has no edges, then $\eps_k^T(S)=0\le 1$. If $S$ has at least one edge, then by Lemma \ref{lemma:token_excess_for_stars} and Corollary \ref{cor:token_isolated}, we obtain $\eps_k^T(S)=1$, as wanted.
\end{proof}

It follows from Lemma \ref{lemma:token_subadditive}, Corollary \ref{cor:token_star_plus_isolated}, and Lemma \ref{lemma:subadditive_cover_bound} that $\eps_k^T(G)\le \tau(G)$ for all graphs $G=(V,E)$ and all $1\le k\le |V|$. In particular we obtain $\eps_k^T(G)\le n-1$ for every $n$-vertex graph $G=(V,E)$ and $1\le k\le n$. Moreover, if $G$ is bipartite, then $\eps_k^T(G)\le \lfloor n/2\rfloor$  (a fact first proved in \cite[Corollary 4]{apte2025conjectured}).

The next lemma is an analogue, in the context of token graphs, of Lemma \ref{lemma:Mk_for_matchings}. Since the argument is similar to that of Lemma \ref{lemma:Mk_for_matchings}, we omit the proof.

\begin{lemma}\label{lemma:token_excess_for_matchings}
Let $t\ge k\ge 1$ and let $V$ be a finite set with $|V|\ge 2t$. Let $G\in \mathcal{M}_t(V)$. Then, 
    \[
        \eps_k^T(G)= 2k-t.
    \]
 In particular, if $t\ge 2k$, $\eps_k^T(G)\le 0$.
\end{lemma}

\begin{proof}[Proof of Theorem \ref{thm:main_token}]
By Lemma \ref{lemma:token_subadditive}, Corollary \ref{cor:token_star_plus_isolated}, Lemma \ref{lemma:token_excess_for_matchings}, and Proposition \ref{prop:star_bound}, we obtain, for every graph $G=(V,E)$ and $1\le k\le |V|$,
\[
    \eps_k^T(G)\le 2(2k-1)=4k-2.
\]
Moreover, if $G$ is bipartite, then
\[
    \eps_k^T(G)=2k-1.
\]
\end{proof}

\section{Concluding remarks}\label{sec:concluding_remarks}

For a graph $H$ and $n\ge 1$, let $\text{ex}(n;H)$ be the maximum number of edges in an $n$-vertex $H$-free graph. It follows from Theorem \ref{thm:main_density_version} that, for every $H$-free graph $G=(V,E)$ and $1\le k\le |V|$, $\eps_k(G)\le \text{ex}(k;H)+(4k-2)\sqrt{k}$. A similar bound, $\eps_k(G)\le \text{ex}(2k;H)$, follows from \cite[Proposition 5.1]{lew2025sums}. We propose the following extension of Conjecture \ref{conj:brouwer}.

\begin{conjecture}\label{conj:brouwer_H_free}
    Let $k\ge 1$, and let $H$ be a graph.  Then, for every $H$-free graph $G=(V,E)$ with $|V|\ge k$,
    \[
        \eps_k(G)\le \text{ex}(k+1;H).
    \]
\end{conjecture}
Note that the bound in Conjecture \ref{conj:brouwer_H_free} cannot be improved. Indeed, let $G=(V,E)$ be obtained from a $(k+1)$-vertex $H$-free graph with $\text{ex}(k+1;H)$ edges by adding $n-k-1$ isolated vertices. Then, clearly, $\sum_{i=1}^k \lambda_i(L(G))= 2|E|= |E|+\text{ex}(k+1;H)$. That is, $\eps_k(G)= \text{ex}(k+1;H)$.

In this paper, we obtained new upper bounds on $\eps_k(G)$ by studying the parameter $\teps_k(G)$. It would be interesting to have a better understanding of the limitations of this approach. More concretely, it would be interesting to determine an optimal upper bound on $\teps_k(G)$.

\begin{problem}
    Let $k\ge 1$. Find the smallest possible $f(k)$ such that $\teps_k(G)\le f(k)$ for all $G=(V,E)$ with $|V|\ge k$.
\end{problem}

The following conjecture was proposed in \cite{apte2025conjectured}.

\begin{conjecture}[Apte, Parekh, Sud \cite{apte2025conjectured}]
\label{con:matching}
    Let $G=(V,E)$ and let $1\le k\le |V|$. Then, $\eps_k^T(G)\le \nu(G)$.
\end{conjecture}
Conjecture \ref{con:matching} would imply, by a simple adaptation of our arguments, that $\eps_k^T(G)\le 2k-1$ for every graph $G=(V,E)$ and $1\le k\le |V|$. Moreover, it would imply $\eps_{k}^T(G)\le \lfloor |V|/2\rfloor$, resolving Conjecture \ref{conj:token} in the special case $k=\lfloor |V|/2\rfloor$. In fact, Apte, Parekh, and Sud showed in \cite[Lemma 13]{apte2025conjectured} that the bound $\eps_k^T(G)\le \lfloor |V|/2\rfloor$ would imply Conjecture \ref{con:matching}, so the two statements are in fact equivalent. Finally, let us also mention that a similar conjecture for $\eps_k(G)$ was proposed in \cite[Conjecture 5.1]{lew2024partition}.

\subsubsection*{Acknowledgments}

We thank James Sud for introducing us to \cite{apte2025conjectured} and in particular to Conjecture \ref{conj:token}, whose study served as the starting point of this work.

\bibliographystyle{abbrv}
\bibliography{biblio}

\end{document}